\documentclass[a4paper,10pt]{amsart}
\usepackage[plainpages=false]{hyperref}
\usepackage{amsfonts,amssymb,amsthm, mathrsfs, lscape}
\usepackage{verbatim}

\usepackage{latexsym}
\usepackage{amssymb}
\usepackage{graphics}
\usepackage{graphicx}
\usepackage[space]{cite}

\usepackage{latexsym}
\usepackage{amsfonts}
\usepackage{amsmath}
\usepackage{latexsym}
\usepackage{amsfonts}
\usepackage{amssymb}
\usepackage{rawfonts}

 \usepackage[usenames,dvipsnames]{pstricks}
 \usepackage{epsfig}
 \usepackage{pst-grad} 
 \usepackage{pst-plot} 
 \usepackage[space]{grffile} 
 \usepackage{etoolbox} 
 \makeatletter 

\renewcommand{\epsilon}{\varepsilon}

\newcommand{\kahler}{K\"ahler }
\newcommand{\horman}{H\"ormander }

\newcommand{\PP}{{\mathbb P}}

\newcommand{\R}{{\mathbb R}}
\newcommand{\C}{{\mathbb C}}

\newcommand{\Z}{{\mathbb Z}}

\newcommand{\vol}{{\operatorname{Vol}}}

\newcommand{\Pic}{{\operatorname{Pic}}}
\newcommand{\Ric}{{\operatorname{Ric}}}

\renewcommand{\phi}{\varphi}

\newcommand{\hcal}{\mathcal{H}}

\newcommand{\ocal}{\mathcal{O}}

\newtheorem{thm}{Theorem}[section]
\newtheorem{theorem}{{Theorem}}[section]

\newtheorem{cor}[theorem]{{Corollary}}
\newtheorem{lem}[theorem]{{Lemma}}
\newtheorem{prop}[theorem]{{Proposition}}

\newenvironment{rem}{\medskip\noindent{\it Remark:\/} }{\medskip}

\newtheorem{nota}[thm]{Notation}
\theoremstyle{definition}

\numberwithin{equation}{section}

\def \C {\mathbb C}
\def \Z {\mathbb Z}
\def \R {\mathbb R}

\def \P {\mathbb P}

\def \p {\partial}
\def \bp {\bar{\partial}}
\def \pz {\partial_z}
\def \bpz {\bar{\partial}_z}
\def \py {\partial_y}
\def \bpy {\bar{\partial}_y}

\title[Projective Line bundle and K-stability]{projective embedding of pairs and logarithmic K-stability}

\author{Jingzhou Sun}
\thanks{The author is partially supported by the Shantou University Start-up funds for Scientific Research}
\address{Department of Mathematics, Shantou University, Shantou City, Guangdong Province 515063, China}
\email{jzsun@stu.edu.cn}

\begin{document}

\begin{abstract}
Let $\hat{L}$ be the projective completion of an ample line bundle $L$ over $D$, a smooth projective manifold. Hwang-Singer \cite{HwangS} have constructed complete CSCK metric on $\hat{L}\backslash D$. When the corresponding \kahler form is in the cohomology class of a rational divisor $A$ and when $L$ has negative CSCK metric on $D$, we show that the Kodaira embedding induced by orthonormal basis of the Bergman space of $kA$ is almost balanced. As a corollary, $(\hat{L},D,cA,0)$ is K-semistable.
\end{abstract}

\maketitle


\section{Introduction}
This is a continuation of our joint work \cite{ssun} with Song Sun about the K-stability of polarized log pairs of \kahler manifolds $(X,D,L)$ with complete constant scalar curvature \kahler (CSCK) metrics on the complement $X\backslash D$, where $D$ is a divisor of $X$ and $L$ is an ample line bundle over $X$. We refer the readers to \cite{ssun} for the backgrounds and motivations of this research, and also for the references related to this research. In \cite{ssun}, we have shown that a polarized log Riemann surface with complete constant negative scalar curvature \kahler metric is almost balanced for large tensor power of the line bundle, hence log K-semistable. 

This article tries to generalize this result to higher dimension. More specifically, in this article we are able to do it in the line bundle case $(\hat{L},D)$, namely when $X$ is the total space $\hat{L}$ of the projective completion of an ample line bundle $L$ of $D$, with a special polarization $A$ that admits a circle-invariant complete negative CSCK metric on the complement $\hat{L}\backslash D$, as constructed by Hwang-Singer in \cite{HwangS} using Calabi ansatz. We will explain more details about this construction in section  \ref{sec2}.


In order to explain our main result,  we have to recall again some known facts and fix some notation.  Let $V$ be a subvariety of $\C\P^N$ and $W$ a subvariety of $V$. For $\lambda\in [0, 1]$ we define the \emph{$\lambda$-center of mass} of $(V, W)$ to be 

$$\mu(V, W, \lambda)=\lambda \int_V \frac{ZZ^*}{|Z|^2} d\mu_{FS}+(1-\lambda)\int_W \frac{ZZ^*}{|Z|^2}d\mu_{FS}-\frac{\lambda Vol(V)+(1-\lambda)Vol(W)}{N+1}Id $$
where $[Z]\in \C\P^N$ is viewed as a column vector, and the volume is calculated with respect to the induced Fubini-Study metric. Notice $\mu$ always takes value in $Lie(SU(N+1))$; indeed, by general theory $\mu$ can be viewed as the moment map for the action of $SU(N+1;\C)$ on a certain \emph{Chow variety}. For $B\in Lie(SU(N+1))$  we write $\|B\|_2:=\sqrt{Tr BB^*}$.

A pair $(V, W)$ embedded in $\C\P^N$ with vanishing $\lambda$-center of mass is called a \emph{$\lambda$-balanced embedding}. We say $(V, W)$ is  \emph{$\lambda$-Chow stable} if there is an $A\in SL(N+1;\C)$ such that $(A.V, A.W)$ is $\lambda$-balanced.  and we say $(V, W)$ is \emph{$\lambda$-Chow semistable} if the infimum balancing energy 
$$E(V, W, \lambda):=\inf_{A\in SL(N+1;\C)} \|\mu(A. V, A. W, \lambda)\|_2$$
vanishes. When $\lambda=1$ the subvariety $W$ can be ignored and this reduces to the standard notion of Chow (semi)-stability. 

Now going back to the situation of a polarized manifold $(X, L)$ and a smooth divisor $D$. We say $(X, D, L)$ is $\lambda$-\emph{almost asymptotically Chow stable} if for all $k$ sufficiently large, under the projective embedding of $(X, D)$ induced by sections of $H^0(X, L^k)$ we have $E(V, W, \lambda)=o(k^{-1})$. By \cite{ssun} if $(X, D, L)$ is $\lambda$-amost asymptotically Chow stable then $(X, D, L, \beta)$ is K-semistable for $\beta=\frac{3\lambda-2}{\lambda}.$  We will not explicitly make use of the notion of (logarithmic) K-(semi)stability in this article, so we will not elaborate on the definition and we refer the readers to \cite{ssun}.

In our setting $D$ is a smooth \kahler manifold of dimension $n$. $L$ is an ample line bundle over $D$ with a Hermitian metric whose curvature defines a \kahler form associated with a negative CSCK metric. We denote by $\omega$ the  \kahler form that defines the complete circle-invariant negative CSCK metric on $\hat{L}\backslash D$, and by $h$ the singular Hermitian metric on $A$ over $\hat{L}$ whose curvature current is $-i\omega$.
For any positive integer, we denote by $\mathcal H_k$ the Bergman space of $L^2$ integrable sections in $H^0(\hat{L},kA)$ under the metric $h^k$ and the volume form $\omega^{n+1}$.

 For $k$ large, we have an embedding $\Phi_k: \hat{L}\rightarrow\P\mathcal H_k^*$. A choice of an orthonormal basis of $\mathcal H_k$ determines a Hermitian isomorphism of $\P\mathcal H_k^*$ with $\C\P^{N_k}$, up to an  $U(N_k+1)$ action, where $N_k+1=\dim \mathcal H_k$. In particular, the $L^2$ norm $\|\mu(\hat{L}, D, \lambda)\|_2$ is independent of the choice of orthonormal basis.   The following is our main result

\begin{thm} \label{thm1-1}
$(\hat{L}, D, A)$ is $\frac{2}{3}$-almost asymptotically Chow stable. More precisely, we have 
$$\|\mu(\Phi_k(\hat{L}), \Phi_k(D), \frac 23)\|^2_2=O(k^{-3/2+n}(\log k)^{121}). $$
\end{thm}

We remark here that the exponent $121$ suffices for our purpose, but it is far from sharp. It can be improved if necessary.

\

An immediate corollary, using formula 2.2 in \cite{ssun}, is that 

\begin{cor}\label{1.1}
$(\hat{L},D,cA,0)$ is K-semistable.
\end{cor}

It is easy to notice that the estimation in \ref{thm1-1} is very similar to that in the main result in \cite{SunSun}, as they share the same basic ideas. The main difference is that in this article we do not have clean explicit formulas of the coefficients of the power series as in \cite{SunSun}. That of course caused a lot of complexity of calculations, which is inevitable, and the conquest of which is the main purpose of this article. It is interesting to notice that in order to get our main result we do not need to estimate the asymptotic of the Bergman kernel near the singularity. But as in the case of log Riemann surface \cite{PunctureD}, we believe that the techniques developed in this article can be used to estimate the Bergman kernel near the singularity $D$. 
\

The structure of this article is as following. We will first quickly review the constructions of the circle-invariant CSCK metrics on the subsets of the total space $\hat{L}$, and introduce our polarization $A$. Then we will do some algebraic calculations in order to calculate the diagonal of the center of mass of $\hat{L}$ and $D$ under the moment map to $Lie(U(N+1))$. Then we will calculate the $L^2$ norm of the center of mass of the log pair, following the basic ideas from \cite{SunSun}, which divides the computation into three components, the ``inside", the ``outside" and the ``neck". We want to emphasize again that the main reason that we can do this is the mass concentration phenomenon, which seems to be due to the negative curvature..

\

\textbf{Acknowledgements.}  The author would like to thank Professor Song Sun for many insightful discussions regarding the topic of this article over a long time. Professor Xiuxiong Chen has shown a lot of interests in and supports for this research, which are greatly appreciated.  The author also want to thank Professor Bernard Shiffman for his continuous and unconditional support.

\section{Circle Invariant CSCK metrics on Line Bundles}\label{sec2}

Let $L\to D$ be an ample line bundle, $n=\dim D$. Let $\hat{L}=L\cup M=\PP(L\oplus \ocal)$ be the projective completion of $L$, with $M$ the divisor at infinity. 
Let $\ocal(1)\to \hat{L}$ denote the natural bundle arises from the projectivelization.

The readers are refered to \cite{HwangS} for details of the momentum construction of the complete circle invariant CSCK metric on $\hat{L}\backslash D$. Here we merely define notations and quickly dive into calculations for our purpose.

Now $\hat{L}\backslash D$ is the total space of the line bundle $p:L^{-1}\to M$. Let $h_M$ be a metric on $L$ with curvature form $i\Theta_{h_M}=\omega_M$, with $\omega_M$ a \kahler form which defined a Riemann metric with constant scalar curvature. Let $t$ be the logarithm of the fibrewise norm function defined by $h_M$ and consider the Calabi ansatz $$\omega=p^*\omega_M+2i\partial\bar{\partial}f(t)$$
with $f$ a function to be chosen.

Let $X$ be the generator of the natural $S^1$ action on $L^{-1}$, normalized so that $exp(2\pi X) = 1$.
Denote by $\tau$ the corresponding moment map determined (up to an additive constant) by
$i_X\omega = −d\tau$. The function $||X||^2_{\omega}$ is constant on the level sets of $\tau$, so there is a function $\phi : I \to (0,\infty)-$to be called the momentum profile of $\omega-$such that $$\phi(\tau)=||X||^2_{\omega}$$
Here $I\subset (-1,\infty)$ is an interval. 

The \kahler form determined by $\phi$ is $\omega_{\phi}=p^*\omega_M+2i\partial\bar{\partial}f(t)=(1+\tau)p^*\omega_M+\phi dt\wedge d^ct$, here we are following the convention used in \cite{HwangS} by defining $d^c=i(\bar{\partial}-\partial$. The corresponding Riemannian metric is denoted by $g_\phi$. 
 
The condition of $\omega_{\phi}$ having constant scalar curvature $c$ is then a second order ODE. And in order for the metric to be extended smoothly through the zero section $M$, the initial data for $\phi$ should be $\phi(0)=0, \phi'(0)=2$. Then $\phi$ is defined by $$\phi(\tau)=\frac{2}{Q(\tau)}(\tau+\int_0^{\tau}(\tau-x)(R(x)-c)Q(x)dx)$$
where $Q(\tau)=(1+\tau)^n$ and $R(\tau)=\frac{S_M}{1+\tau}$, with $S_M$ the constant scalar curvature of Riemannian metric defined by $\omega_M$.

$\phi$ clearly depends on the choice of $c$.
By \cite{HwangS}, there is an unique choise of $c=c_0$ such that $\omega_\phi$ defines a complete metric over the whole $L^{-1}=\hat{L}\backslash D$. And $c_0$ is characterized as the biggest $c$ that makes $\phi(\tau)\geq 0$ for all $\tau\geq 0$. 

The restriction of $g_\phi$ to a fibre is totally geodesic and is given by $$g_{fibre}=\phi(\tau)|\frac{dz}{z}|^2$$
where $z$ is any linear coordinate on the fibre. When $c=c_0$, the smallest positive nonzero root $\tau_0$ of $\phi(\tau)$ is the area of a fibre.

\begin{eqnarray*}
\phi(\tau) &=& \frac{2}{(1+\tau)^n}[\tau+\frac{1}{n(n+1)(n+2)}[(2+n)S_M(-1-(1+n)\tau \\ &+& (1+\tau)^{1+n})+nc(1+(n+2)\tau-(1+\tau)^{n+2}]]
\end{eqnarray*}
Inside the bracket, the term of top degree is $\frac{-c}{(n+2)(n+1)}\tau^{n+2}$. So to make $\phi(\tau)\geq 0$ for $\tau>>0$, we must have $c\leq 0$. 
Let $c=0$, then
 $$\phi(\tau) = \frac{2}{(1+\tau)^n}[\tau+\frac{1}{n(n+1)}S_M\sum_{i=2}^{n+1}\begin{bmatrix}
n+1\\ i
\end{bmatrix}\tau^i]$$
so when $S_M\geq 0$, we always have $\phi(\tau)\geq 0$ for $\tau\geq 0$. Therefore $c_0=0$. And in this case $\phi(\tau)>0$ for all $\tau>0$, so the complete metric will have infinite fibre area.

In this article we will assume $S_M< 0$, so we must have $c_0<0$.

When $n=1$, we have
$$\phi(\tau)=\frac{2\tau}{1+\tau}(-\frac{c}{6}\tau^2+\frac{S_M-c}{2}\tau+1)$$
Clearly, $c_0$ is then a solution of the equation $$\Delta=(\frac{S-c}{2})^2+\frac{4c}{6}=0.$$
Since $c_0$ must make $\phi$ have a positive zero of multiplicity $>1$, we must also have $S_M-c_0<0$. Therefore we have 
$$c_0=S_M-\frac{4}{3}+\frac{2}{3}\sqrt{4-6S_M}$$
and $$\phi_{c_0}(\tau)=\frac{-c_0\tau}{3(1+\tau)}(\tau-\tau_0)^2$$

\begin{theorem}
$$\phi_{c_0}(\tau)=\frac{2\tau}{(1+\tau)^n}(\tau-\tau_0)^2f(\tau)$$
where $\tau_0>0$ and $f$ is a polynomial satisfying $f(\tau)>0$ for $\tau\leq \tau_0$
\end{theorem}
\begin{proof}
We already know that the numerator part $\bar{\phi}(\tau)$ of $\phi_{c_0}(\tau)$ vanishes at $\tau_0$ with multiplicity $\geq 2$. When $n=2$, we factorize $\bar{\phi}(\tau)=\lambda\tau (\tau-\tau_0)^2(\tau-\tau_1)$. It is easy to check that $\tau_1<0$. 

In general, we first notice that since $\bar{\phi}(\tau)\geq 0$ for $\tau\geq 0$. So the multiplicity of $\bar{\phi}(\tau)$ at $\tau_0$ must be even.

take the second derivative of $\bar{\phi}$,
$$\bar{\phi}''(\tau)=(1+\tau)^{n-1}(S_M-c_0(1+\tau))$$
It is now clear that $\bar{\phi}''$ can not vanish at $\tau_0$ with multiplicity $>1$. The theorem is now proved.
\end{proof}

We now show that this vanishing order means that the metric is of Poincare type along the fibres.

The relations between $\tau$ and $t$ are:
\begin{eqnarray*}
t&=&\int_{\bar{\tau}}^\tau \frac{dx}{\phi(x)}\\
f(t) &=& \int_{\bar{\tau}}^\tau \frac{xdx}{\phi(x)}\\
\tau &=& f'(t)
\end{eqnarray*}

\begin{nota}
In this article, when we only need a rough picture, we will use the notation $\sim$, which will mean asymptocally, up to a positive scalar multiple, the two sides are the same.
\end{nota}
Now since $\phi(\tau)$ vanishes to the second order at $\tau_0$, we have $t\sim \frac{1}{\tau_0-\tau}$ when $\tau$ approaches $\tau_0$. So $\phi(\tau)\sim t^{-2}$. Therefore $g_{fibre}\sim \frac{1}{(\log |z|^2)}|\frac{dz}{z}|^2$. The right hand side of $\sim$ is just the Poincare metric on the punctured disk.

Because the metric $\omega_\phi$ has finite fibre area, it must have finite mass near $D$. So by El Mir extension theorem \cite{DemaillyComplex}, $\omega_\phi$ extends by $0$ to be a closed $(1,1)-$current over $\hat{L}$. By abuse of notation, we still use $\omega_\phi$ to denote the extended current. We let $\pi:\hat{L}\to D$ denote the projection.
\begin{lem}
The cohomology class of $\omega_\phi$ on $\hat{L}$ is $c_1((1+\tau_0)\pi^*L+\ocal(\tau_0))=c_1((1+\tau_0)[D]-\ocal(1))$.
\end{lem}
\begin{proof}
We have $\Pic(\hat{L})=\Pic(D)\oplus \Z$, and $H^{1,1}(\hat{L},\R)=H^{1,1}(D,\R)\oplus \Z$. So $[\omega_\phi]=\pi^*\alpha+c_1(\ocal(b))$, where $\alpha\in H^{1,1}(D,\R)$.
Since the fibre area is $\tau_0$, we have $b=\tau_0$. 

We also know that $\omega_\phi|_M=\omega_M\in c_1(L)$ and $\ocal(1)|_M=-L$. So $\alpha=c_1((1+\tau_0)L)$. Then we use the fact that $[D]|_D=L$ and $\ocal(1)|_D=\ocal$ to get that $[D]=\pi^*L+\ocal(1)$.
\end{proof}
We let $A=(1+\tau_0)[D]-\ocal(1)$. Clearly $\tau_0$, depending on $D$ and $L$, may not be rational, which means $c\omega_\phi$ is in the first Chern class of a line bundle.  

In the remaining of this article, we will assume that $\tau_0$ is rational, namely for some integer $c$, $cA$ is a line bundle on $\hat{L}$, which is clearly ample. Then there is a singular metric $h$ on $cA$ with curvature current
$c\omega_\phi$. The choice of $h$ is unique up to a scalar multiplication. 
Then as in one dimension, we consider the Bergman space $\hcal_k\subset H^0(\hat{L},kcA)$, the space of $L^2$ integrable sections with respect to the metric $h$ and the volume form $\frac{\omega_\phi^{n+1}}{(n+1)!}$. We can use $K$ to denote $kc$, but since in the estimations following, we will simply not use the fact that $k$ is an integer, so for simplicity we just safely assume $c=1$.

We denote by $\Phi_k:\hat{L}\to \PP \hcal_k^*$ the Kodaira embedding defined by an orthonormal basis of $\hcal_k$.

\section{Algebraic Calculations}
Write $A_k=kA-D$.
$\dim \hcal_k=h^0(\hat{L},A_k)=\chi(\hat{L},A_k)$.
By Hirzebruch-Riemann-Roch, 
$$\chi(\hat{L},A_k)=\frac{A_k^{n+1}}{(n+1)!}+\frac{1}{2n!}A_k^n\cdot c_1(\hat{L})+lower\quad terms$$

Since $D|_D=L$, we have $$D^{n+1}=L^n$$
where by $L^n$, we mean the self intersection number of $L$ on $D$.
Since $\ocal(1)|_D=\ocal_D$, we have $$D^a\cdot \ocal(1)^b=0$$
for $a>0,b>0, a+b=n+1$. 
Since $\ocal(1)=[M]$ and $\ocal(1)|_M=-L$, we have $$\ocal(1)^{n+1}=(-1)^nL^n$$
Therefore, we have $$A_k^{n+1}=[(k(1+\tau_0)-1)^{n+1}-k^{n+1}]L^n$$

We write $\hat{K}=K_{\hat{L}}$, the canonical bundle. Then $(\hat{K}+D)|_D=K_D$. So 
$\hat{K}|_D=K_D-D$. So $\hat{K}=\pi^*(K_D-L)+\ocal(a)$.
Similarly, we have $\hat{K}|_M=K_M+L$. So we have $a=-2$.
So $$\hat{K}=\pi^*(K_D-L)+\ocal(-2)$$
So $$c_1(\hat{L})=\pi^*(L-K_D)+\ocal(2)$$
\begin{eqnarray*}
A_k^n\cdot c_1(\hat{L})&=&(kA-D)^n\cdot \ocal(2)+(kA-D)^n\cdot \pi^*(L-K_D)\\
&=&2k^n L^n+[(k(1+\tau_0)-1)^n-k^n]L^{n-1}\cdot (L-K_D)
\end{eqnarray*}

So if we write out $\dim \hcal_k$ asymptocally in $k$ as 
$$\dim \hcal_k=\frac{1}{(n+1)!}(a_0k^{n+1}+a_1k^n+O(k^{n-1}))$$
then \begin{eqnarray*}
a_0&=&[(1+\tau_0)^{n+1}-1]L^n\\
a_1&=&-\frac{n+1}{2}[(1+\tau_0)^{n}-1](L^n+L^{n-1}\cdot K_D)
\end{eqnarray*}

We then calculate the volue of $\Phi_k(\hat{L})$ and $\Phi_k(D)$ as subvarieties in $\PP \hcal_k^*$. We have 
\begin{eqnarray*}
\vol \hat{L}&=&\frac{1}{(n+1)!}(\Phi_k^*\ocal(1))^{n+1}=\frac{A_k^{n+1}}{(n+1)!}\\
&=&[(k(1+\tau_0)-1)^{n+1}-k^{n+1}]\frac{L^n}{(n+1)!} 
\end{eqnarray*}
and \begin{eqnarray*}
\vol D&=&\frac{(\Phi_k^*\ocal(1))^{n}}{n!}=\frac{(A_k|_D)^{n}}{n!}\\
&=&(k(1+\tau_0)-1)^{n}\frac{L^n}{n!}
\end{eqnarray*}
So
$$\vol \hat{L}+\frac{1}{2}\vol D=\frac{1}{(n+1)!}(b_0K^{n+1}+b_1k^n+O(k^{n-1}))$$
where \begin{eqnarray*}
b_0&=&[(1+\tau_0)^{n+1}-1]L^n\\
b_1&=&-\frac{n+1}{2}(1+\tau_0)^{n}L^n
\end{eqnarray*}

So asymptotically
$$\frac{\vol \hat{L}+\frac{1}{2}\vol D}{\dim \hcal_k}=1+\sigma \frac{1}{k}+O(\frac{1}{k^2})$$
where $\sigma=\frac{b_1-a_1}{a_0}$.
We also have $$b_1-a_1=\frac{n+1}{2}[((1+\tau_0)^n-1)L^{n-1}\cdot K_D-L^n)$$

On the other hand, we have $$\int_{L^{-1}} c_0\omega_\phi^{n+1}=(n+1)\int_{L^{-1}} \Ric(\omega_\phi)\omega_\phi^{n}$$

Since the whole Ricci current $\Ric(\omega_\phi)=\Ric(\omega_\phi)|_{L^{-1}}+[D]$, we have $$c_1(\hat{L})\cdot A^n=D\cdot A^n+\frac{1}{n+1}\int_{L^{-1}} \Ric(\omega_\phi)\omega_\phi^{n}$$
Therefore $$\frac{c_0}{n+1}\int_{L^{-1}} \omega_\phi^{n+1}=L^n-((1+\tau_0)^n-1)L^{n-1}\cdot K_D$$
So we have 
\begin{lem}
\begin{equation}
\frac{c_0}{2}=-\sigma
\end{equation}
\end{lem}

Consider the exact sequence 
\begin{eqnarray*}
0&\to & H^0(\hat{L},kA-(a+1)D)\to H^0(\hat{L},kA-aD)\to H^0(D,kA-aD)\\ &\to & H^1(\hat{L},kA-(a+1)D)\to\cdots
\end{eqnarray*}
Since $\hat{K}=\pi^*(K_D-L)+\ocal(-2)$, we have 
$$H^1(\hat{L},kA-(a+1)D)=H^1(\hat{L},\hat{K}+L_a)$$
where $L_a=\pi^*((k(1+\tau_0)-a)L-K_D)+\ocal(k\tau_0-a+1))$. And when $a<k\tau_0+1$, $L_a$ is ample. Therefore by Kodaira vanishing theorem,  $$H^1(\hat{L},kA-(a+1)D)=0$$
But when $a=k\tau_0+1$, $(kA-aD)|_{fibre}=\ocal(-1)$, so $H^0(\hat{L},kA-aD)=0$.
Therefore, we have complete decomposition
$$H^0(\hat{L},kA-D)=\bigoplus_{a=1}^{k\tau_0}H^0(D,kA-aD)$$
Actually, let $s_M\in H^0(\hat{L},\ocal(1))$ be a section whose zero set is $M$,
we can choose a natural injection $$j_a:H^0(D,kA-aD)\to H^0(\hat{L},kA-D)$$ 
with $j_a(s_a)=\pi^*(s_a)\otimes s_M^{k\tau_0-a}$
\

\section{Local Data}  

\subsection{Basic Settings}
We will denote by $\epsilon (k)$ a quantity depending on $k$ that is $O(k^{-m})$ as $k\rightarrow\infty$, for all $m\geq 0$. 
For simplicity, when we have equation of the form $a=(1+\epsilon(k))b$, we will simply write $a=b$.

We will also denote by $L^{-1}$ the open manifold $\hat{L}\backslash D$ and by $L$ the open manifold $\hat{L}\backslash M$. It is worth to mention that the gluing map is just $v\mapsto v^{-1}$.

Recall $h_M$ is the metric of $L$ on $M$.  
Let $e_L$ be a local frame for $L$ on $M$ with $|e_L|^2_{h_M}=e^{-\phi_M}$. 
Then $p^*e_L$ is a local frame of $\pi^*L$ on $\hat{L}$.

There is an unique, up to scalar multiplication, section $s_D\in H^0(\hat{L},[D])$ such that the zero set $Z_{s_D}=D$. On $L$, $s_D$ is a section of $\pi^*L$. In a neighborhood $U_p$ of  a point $p\in D$, using the frame $e_L$, $L$ has local coordinates $(w,z)$, where $w$ is local coordinates for $D$. We will always choose $U_p$ to be of the form $\pi^{-1}(V_p)$, where $V_p$ is a neighborhood of $p$ in $D$.
Then under the frame $p^*e_L$, 
$s_D$ is a function $\lambda z$. We choose and fix $s_D$ so that the $\lambda$ is $1$. Namely, locally $s_D(ze_L)=zp^*e_L$. 

$e_L\otimes s_D^{\tau_0}$ is now a local frame for $A$ on $L^{-1}$. And we set $$|p^*e_L\otimes s_D^{\tau_0}|_h=e^{-\phi_M-2f(t)}$$
 
In $U_p$, $p^*e_L^{1+\tau_0}$ is a local frame for $A$. One can easily check that we have to set in $U_p$, $$|p^*e_L^{1+\tau_0}|_h=e^{-\phi_M-2f(t)-\tau_0\log |z|^2}$$
in order to be compatible with the norm of $p^*e_L\otimes s_D^{\tau_0}$, since $s_D=zp^*e_L$.

Clearly, the metric $h$ defined above has curvature form $\omega_\phi$ on $L^{-1}$.For cohomological reason, we have that the curvature current of $h$ is just the extended $\omega_\phi$.

Now $$\omega_\phi^{n+1}=(n+1)(1+\tau)^np^*\omega_M^n\phi dt\wedge d^ct$$
so the volume form is $$\frac{\omega_\phi^{n+1}}{(n+1)!}=(1+\tau)^np^*\frac{\omega_M^n}{n!}\phi dt\wedge d^ct$$

So for $s\in \hcal_k$, in $U_p$, $s=q(w,z)p^*e_L^{k(1+\tau_0)}$. The $L^2$ integral is given by $$\int |q|^2e^{-kg}(1+\tau)^n\phi(\tau)|\frac{dz}{z}|^2\frac{\omega_M^n}{n!}$$
where $g=\phi_M+2f(t)+\tau_0\log |z|^2$.

We have $t=-\log |z|+\frac{\phi_M}{2}$. So 
$$g=(1+\tau_0)\phi_M+2(f(t)-\tau_0 t)$$
We can then estimate 
\begin{eqnarray*}
f(t)-\tau_0 t&=&\int_{\bar{\tau}}^\tau\frac{x-\tau_0}{\phi(x)}dx\\
&\sim &\log (\tau_0-\tau)
\end{eqnarray*}
so asymptocally $g\sim \log (\tau_0-\tau)$.

Now it is clear that the $L^2$ sections are exactly those vanishing along $D$, therefore
 $\hcal_k$ is the image of the embedding map 
$$\otimes_{s_D}:H^0(\hat{L},kA-D)\to H^0(\hat{L},kA)$$
By abuse of notation, we will still denote by $j_a$ the composition $\otimes_{s_D}\circ j_a$.

In order to produce an orthonormal basis for $\hcal_k$, we notice the circle symmetry of the metric and the volume form, which means the following 
\begin{lem} 
Under the embeddings $j_a$, the decomposition
$$\hcal_k=\bigoplus_{a=1}^{k\tau_0}H^0(D,kA-aD)$$
is orthogonal.
\end{lem}
This lemma means that we do not need to use the \horman $L^2$ techniques to construct global sections using local sections.

For the space $H^0(D, kA-aD)=H^0(D, [k(1+\tau_0)-a]L)$, we have a natural Hermitian metric induced by $h_M$.
Given an unit section $s_a\in H^0(D, kA-aD)$, we can calculate the $L^2$ integral
$$\int_L |z|^{2a}|s_a|^2e^{-kg}(1+\tau)^n\phi(\tau)|\frac{dz}{z}|^2\frac{\omega_M^n}{n!} $$
where by abuse of notation, we use $s_a$ to denote its local representive under a local frame $e_L$.

We integrate along the fibres first by choosing a local frame for each point $p\in D$ such that $\phi_M(p)=0$. Then the integral is the same for every fibre:
\begin{equation}
I_a=\int_{\C} |z|^{2a-2}(1+\tau)^ne^{-kg}\phi(\tau)|dz|^2
\end{equation}
So the normalized $\bar{s_a}=\frac{j_a(s_a)}{\sqrt{I_a}}$ is of norm $1$ in $\hcal_k$.
\subsection{Estimations for $I_a$}
We assume that $a<<k$.

We have $t=-\log |z|$. So 
\begin{equation}
I_a=2\pi\int_{-\infty}^{\infty}e^{E_a(t)}dt
\end{equation}
where $$E_a(t)=-2at+n\log(1+\tau)-kg(\tau)+\log \phi(\tau)$$
We want to show that the integral is Laplace type, namely the Laplace method can be used to estimate it. So
noticing that $$\frac{dt}{d\tau}=\frac{1}{\phi(\tau)}$$
$$\frac{d}{dt}=\frac{d\tau}{dt}\frac{d}{d\tau}=\phi(\tau)\frac{d}{d\tau}$$
$$\frac{dg}{d\tau}=\frac{\tau-\tau_0}{\phi(\tau)}$$

we take the derivatives: 
$$\frac{dE_a}{dt}=-2a+\frac{n\phi(\tau)}{1+\tau}-2k(\tau-\tau_0)+\phi'(\tau)$$
$$\frac{d^2E_a}{dt^2}=\phi(\tau)(n\frac{(1+\tau)\phi'(\tau)-\phi(\tau)}{(1+\tau)^2}-2k+\phi'')$$

So it is clear that for $k$ big enough, we have $$\frac{d^2E_a}{dt^2}<0$$ for all $t$
namely $E_a(t)$ is a concave function. 

It is easy to see that for $\tau$ far from $\tau_0$, $\frac{dE_a}{dt}>0$, and for $\tau$ close enough to $\tau_0$, $\frac{dE_a}{dt}<0$. Therefore, $E_a$ has an unique critical point $t_a$. We also write $\tau_a=\tau(t_a)$.

Letting $\frac{dE_a}{dt}=0$, we can estimate $\tau_a$, seeing $$2k(\tau_0-\tau_a)=2a-n\frac{\phi(\tau_a)}{1+\tau_a}-\phi'(\tau_a)$$
As $\phi(\tau)\sim (\tau-\tau_0)^2$ and $\phi'(\tau)\sim \tau-\tau_0$, we see that $$\tau_0-\tau_a=\frac{a}{k}$$
is a good estimation, with error $O(\frac{\tau_0-\tau_a}{k})=O(\frac{a}{k^2})$. So we have $$t_a\sim \frac{k}{a}$$

Then we can estimate $\frac{d^2E_a}{dt^2}(t_a)$: 
$$E_a''(t_a)=-(1+O(\frac{1}{k}))\eta_2(\tau_0)2k(\tau_a-\tau_0)^2$$
where $\eta_2(\tau)=\frac{\phi(\tau)}{(\tau-\tau_0)^2}$.
So $$E_a''(t_a)=-(1+O(\frac{1}{k}))2\eta_2(\tau_0)\frac{a^2}{k}$$
and we know that $2\eta_2(\tau_0)>0$.

\begin{lem}\label{margin}
Let $f(x)$ be a concave function. Suppose $f'(x_0)<0$, then we have
$$\int_{x_0}^\infty e^{f(x)}dx\leq\frac{e^{f(x_0)}}{-f'(x_0)}$$
\end{lem}
\begin{proof}
Since $f$ is concave, $f'(x)\leq f'(x_0)$
$$-f'(x_0)\int_{x_0}^\infty e^{f(x)}dx\leq \int_{x_0}^\infty e^{f(x)}(-f'(x))dx=e^{f(x_0)}$$
\end{proof}
The idea of the Laplace method is that we can use $\frac{E_a''(t_a)}{2}(t-t_a)^2$ to estimate $E_a(t)-E_a(t_a)$. But before we can use it, we must find a radius such that this estimation is valid.

We have $E_a'(t)=-2k(\tau-\tau_a)(1+O(\frac{1}{k}))$. So going from $\tau_a$ to $\tau_a+\epsilon_\tau$, the difference of $E_a$ is 
\begin{eqnarray*}
\Delta E_a &=&\int_{t_a}^{t_a+\epsilon_t}E_a'(t)dt\\
&=&(1+O(\frac{1}{k}))\int_{\tau_a}^{\tau_a+\epsilon_\tau}\frac{-2k(\tau-\tau_a)}{\phi(\tau)}d\tau
\end{eqnarray*}
Using new variable $y=\tau-\tau_a$, 
\begin{eqnarray*}
\Delta E_a &=&(1+O(\frac{1}{k}))\int_0^{\epsilon_\tau}\frac{1}{\eta_2(y+\tau_a)}\frac{-2ky}{(y-(\tau_0-\tau_a))^2}dy\\
&=&(1+O(\frac{a}{k}))Ck[\frac{\tau_0-\tau_a}{\tau_0-\tau_a-y}+\log (\tau_0-\tau_a-y)]_0^{\epsilon_\tau}
\end{eqnarray*}
where $C<0$. Assuming $\nu=\frac{\epsilon_\tau}{\tau_0-\tau_a}$ being small, we get 
$$\Delta E_a=(1+O(\frac{a}{k}+\nu^3))\frac{Ck}{2}\nu^2$$

In the spirit of lemma \ref{margin}, we need $\Delta E_a\sim (\log k)^2$, then $\nu\sim \frac{\log k}{\sqrt{k}}$. So $\nu$ is indeed small. Then $\epsilon_\tau\sim \frac{a\log k}{k\sqrt{k}}$. Then we can estimate $\epsilon_t$:
\begin{eqnarray*}
\epsilon_t&=&\int_{\tau_a}^{\tau_a+\epsilon_\tau}\frac{1}{\phi(\tau)}d\tau\\
&\sim & \frac{\epsilon_\tau}{(\tau_0-\tau_a)^2}\sim \frac{\sqrt{k}\log k}{a}
\end{eqnarray*}
For the left side of $t_a$, we have basically the same argument.

\begin{rem}
 So lemma \ref{margin} tells us that the mass is concentrated in a $\epsilon_t$ neighborhoof of $t_a$ with an error of $\epsilon(k)$. We have $\epsilon_t<<t_a$, so like the punctured disk case, we also have mass concentration here, which is key to our following calculations. 
\end{rem}
We also need to estimate $E_a^{(3)}(t)$ and $E_a^{(4)}(t)$
\begin{eqnarray*}
E_a^{(3)}(t)&=&\frac{d\tau}{dt}\frac{d}{d\tau}E_a''(t)\\ 
&=&\phi(\tau)[\phi'(-2k-\frac{n}{(1+\tau)^2}+\frac{n\phi'(\tau)}{1+\tau}+\phi'')\\
&+&\phi(\frac{2n}{(1+\tau)^3}-\frac{n\phi'(\tau)}{(1+\tau)^2}+\frac{n\phi''(\tau)}{1+\tau}+\phi^{(3)})]
\end{eqnarray*}
\begin{eqnarray*}
E_a^{(4)}(t)&=&\frac{d\tau}{dt}\frac{d}{d\tau}E_a^{(3)}(t)\\
&=&\phi[\phi'^2[\frac{n\phi'(\tau)}{1+\tau}-2k-\frac{n}{(1+\tau)^2}+\phi'')]+\phi\phi'[\frac{6n}{(1+\tau)^3})+\frac{3n\phi''}{1+\tau}\\
&+&\frac{3n\phi'}{(1+\tau)^2}+3\phi^{(3)})]+\phi^2[-\frac{6n}{(1+\tau)^4}+\phi^{(4)}+\frac{2n\phi'}{(1+\tau)^3}-\frac{2n\phi''}{(1+\tau)^2}]\\
&+&\phi\phi''[-2k-\frac{n}{(1+\tau)^2}+\phi''+\frac{n\phi'}{1+\tau}]]\\
\end{eqnarray*}
So in the $\epsilon_t$ neighborhood of $t_a$, we have $$E_a^{(3)}(t)=(1+O(\frac{1}{k}))Ck\frac{a}{k}\frac{a^2}{k^2}=(1+O(\frac{1}{k}))C\frac{a^3}{k^2}$$
$$E_a^{(4)}(t)=(1+O(\frac{1}{k}))Ck\frac{a^2}{k^2}\frac{a^2}{k^2}=(1+O(\frac{1}{k}))C\frac{a^4}{k^3}$$

\begin{prop}
\begin{equation}\label{apr-Ia}
I_a=(1+O(\frac{1}{k})) 2\pi e^{E_a(t_a)}\sqrt{\frac{2\pi}{- E_a''(t_a) }}
\end{equation}
\end{prop}
\begin{proof}

We can write $I_a=2\pi e^{E_a(t_a)}\int_0^{\infty}e^{E_a''(t_a)(t-t_a)^2/2+b(t-t_a)^3+\delta(t-t_a)}dt$. 
As we have seen that we can change the integral region to $(t_a-\epsilon_t,t_a+\epsilon_t)$. Then we can let $\bar{t}=t-t_a$. 
Then $$\int_0^{\infty}e^{E_a''(t_a)(t-t_a)^2/2+b(t-t_a)^3+\delta(t-t_a)}dt=(1+\epsilon(k))\int_{-\epsilon_t}^{\epsilon_t}e^{\frac{E_a''(t_a)}{2}\bar{t}^2+b\bar{t}^3+\delta(\bar{t}})d\bar{t}$$
with $\delta(\bar{t})=\bar{t}^4\gamma(\bar{t})$. 
By the symmetry of the measure, we can ignore the degree $3$ term. 
By our previous estimation of $E_a^{(4)}(t)$, we get that 
$$\delta(\bar{t})=C\frac{a^4}{k^3}\bar{t}^4(1+O(\frac{1}{k}))$$
Therefore the error is about $$C\frac{a^4}{k^3}\int_{-\infty}^{\infty}x^4e^{\frac{E_a''(t_a)}{2}x^2}dx=C\frac{a^4}{k^3}\frac{3\sqrt{2\pi}}{(-E_a''(t_a))^{5/2}}$$
And the ratio to $(-E_a''(t_a))^{-1/2}$ is $O(\frac{1}{k})$
\end{proof}

We have seen in dimension one that we will need to look at the ratios $\frac{I_a}{I_{a+1}}$. The squre root part is easy, namely
$$\sqrt{\frac{E_{a+1}''(t_{a+1})}{E_{a}''(t_{a})}}\approx \frac{a+1}{a}$$
Now we deal with the main part. For simplicity, we will write $$E_{a,0}=E_a(t_a)$$ $$g_a=g(\tau_a)$$
Then 
\begin{eqnarray*}
E_{a,0}-E_{a+1,0}&=& 2[(a+1)\int_{\bar{\tau}}^{\tau_{a+1}}\frac{dx}{\phi(x)}-a\int_{\bar{\tau}}^{\tau_{a}}\frac{dx}{\phi(x)}]\\
&+& n\log\frac{1+\tau_a}{1+\tau_{a+1}}-k(g_a-g_{a+1})+\log \frac{\phi(\tau_{a})}{\phi(\tau_{a+1})}
\end{eqnarray*}	
For the $g_a$ terms, we have $g_a=2\int_{\bar{\tau}}^{\tau_{a}}\frac{(x-\tau_0)dx}{\phi(x)}$.
Therefore 
\begin{eqnarray*}
g_a-g_{a+1}&=&2\int_{\tau_{a+1}}^{\tau_a}\frac{(x-\tau_0)dx}{\phi(x)}\\
&=&2\int_{\tau_{a+1}}^{\tau_a}\frac{\eta(x)dx}{x-\tau_0}\\
&=&2(1+O(\frac{1}{k}))\eta(\tau_a)\log \frac{\tau_0-\tau_{a}}{\tau_0-\tau_{a+1}}+O(\frac{1}{k})
\end{eqnarray*}
where $\eta(x)=\frac{1}{\eta_2(x)}$. And the last equation is estimated using integration by parts.

To estimate $\int_{\bar{\tau}}^{\tau_{a}}\frac{dx}{\phi(x)}$, we use integration by parts twice to get:
\begin{eqnarray*}
\int_{\bar{\tau}}^{\tau_{a}}\frac{dx}{\phi(x)}&=&\frac{\eta(\tau_a)}{\tau_0-\tau_a}-\frac{\eta(\bar{\tau})}{\tau_0-\bar{\tau}}+\eta'(\tau_a)\log(\tau_0-\tau_a)\\
&-&\eta'(\bar{\tau})\log(\tau_0-\bar{\tau})-\int_{\bar{\tau}}^{\tau_{a}}\eta''(x)\log (\tau_0-x)dx
\end{eqnarray*}	
The last integral converges even if we replace $\tau_a$ by $\tau_0$. 
We denote $$\Delta_t=(a+1)\int_{\bar{\tau}}^{\tau_{a}}\frac{dx}{\phi(x)}-a\int_{\bar{\tau}}^{\tau_{a}}\frac{dx}{\phi(x)}$$.
Then
\begin{eqnarray*}
\Delta_t =\eta(\tau_{a+1})k-\eta(\tau_a)k +(a+1)\eta'(\tau_{a+1})\log\frac{a+1}{k}-a\eta'(\tau_{a})\log\frac{a}{k}+O(\frac{1}{k})
\end{eqnarray*}

Since \begin{eqnarray*}
\eta_(\tau_{a+1})&=&\eta_(\tau_{a})+\eta'(\tau_a)(\tau_{a+1}-\tau_a)+O((\tau_{a+1}-\tau_a)^2)\\ &=&\eta_(\tau_{a})+\eta'(\tau_a)\frac{1}{k}+O(\frac{1}{k^2})
\end{eqnarray*}
we get
\begin{eqnarray*}
\Delta_t=O(\log k)+O(1)=O(\log k)
\end{eqnarray*}
Lastly, $$\log \frac{\phi(\tau_{a})}{\phi(\tau_{a+1})}\approx 0$$
So we have proved the following
\begin{lem}\label{2 terms}
$E_{a,0}-E_{a+1,0}=2k\eta(\tau_a)\log \frac{a+1}{a}+O(\log k)$
\end{lem}
As long as $a=O(\frac{k}{(\log k)^2})$, we always have $e^{E_{a,0}-E_{a+1,0}}>k^N$ asymptotically for all $N$.

We then also have the following
\begin{lem}\label{3 terms}
$e^{2E_{a+1,0}-E_{a,0}-E_{a+2,0}}=\epsilon(k)$ as long as $a=O(\frac{\sqrt{k}}{\log k})$
\end{lem}
\begin{proof}
The main difference is $\Delta_{\eta}=2k\eta(\tau_a)\log \frac{a+1}{a}-2k\eta(\tau_{a+1})\log \frac{a+2}{a+1}$. So
\begin{eqnarray*}
\Delta_{\eta}&=&2k[\eta(\tau_a)\log \frac{a+1}{a}-(\eta(\tau_a)+O(\frac{1}{k})\log \frac{a+2}{a+1})]\\&=&2k\eta(\tau_a)\log \frac{(a+1)^2}{a(a+2)}
\end{eqnarray*}
Therefore $$e^{E_{a,0}+E_{a+2,0}-2E_{a+1,0}}\approx e^{\frac{2k}{a^2}\eta}*e^{O(\log k)}$$
Now as long as $a=O(\frac{\sqrt{k}}{\log k})$, we have $e^{E_{a,0}+E_{a+2,0}-2E_{a+1,0}}>k^N$ asymptotically for all $N$.
\end{proof}

\section{Calculations of $\mu_a$}

\begin{nota}
For simplicity we will fix the notation	$N=k(1+\tau_0)$ 
\end{nota}
The injection $j_a$ keeps orthogonality, namely if we have a pair of orthogonal sections $s_{a,i}\in H^0(D, kA-aD)$, $i=1,2$. Then we have 
$$(\bar{s}_{a,1},\bar{s}_{a,2})_h=0$$
Therefore, let $\{s_{a,i}\}, 1\leq i\leq m_i$ be an orthonormal basis of $H^0(D, kA-aD)$, then we have an orthonormal basis
$\{\overline{s_{a,i}}\}$ of $\hcal_k$.
We will use the notations $$\overline{\rho_a}=\sum_i |\bar{s}_{a,i}|_h^2$$ 
$$\rho_a=\sum_i |s_{a,i}|_{h_M}^2$$
Their relations is the following:
$$\overline{\rho_a}=\frac{\rho_a}{I_a}|z|^{2a}e^{-kg+(N-a)\phi_M}$$
where $N=k(1+\tau_0)$.

We have the Bergman kernel $$\bar{\rho}=\sum \overline{\rho_a}$$
and $$\omega_{FS}=i\p\bp \log\bar{\rho}+k\omega_\phi $$

Since  the fibres are totally geodesic, we can estimate the injective radius and see that in order for the injective radius to be $\geq \frac{\log k}{\sqrt{k}}$, we need $$\frac{1}{\log (1/x)}=O(\frac{\log k}{\sqrt{k}})$$ For simplicity, we just assume $\log \frac{1}{x}\leq \frac{\sqrt{k}}{\log k}$. That means outside the disc bundle $D_r$ of radius $r=e^{(-k^{1/2}\log k) /2}$, we have the asymptotic expansion of the Bergman kernel just the same as the compact case, namely
$$\bar{\rho}=k^{n+1	}(1+\frac{c_0}{2}\frac{1}{k}+O(\frac{1}{k^2}))$$
Our estimations for $I_a$ works for $a<<k$. We need to show that those $\overline{s_{a,i}}$'s with $a$ big, say $a\geq k^{3/4}$, has $\epsilon(k)$ contributions to $\bar{\rho}$ within $D_r$.
\begin{lem}
Let $\rho_{in}=\sum_{a<k^{3/4}} \overline{\rho_a}$.
For $a\geq k^{3/4}$, we have 
$$\overline{\rho_a}=\epsilon(k)\rho_{in}$$
within $D_r$
Therefore, $$\bar{\rho}=(1+\epsilon(k))\rho_{in}$$
within $D_r$
\end{lem}
\begin{proof}
This is a routine. We can use a middle term, for example $a_0=k^{5/8}$, as a bridge. It is clear that $\overline{\rho_{a_0}}=\epsilon(k)\rho_{in}$ within $D_r$. 

The norm along a fibre over $p$, with $\phi_M(p)=0$, is of the form $e^{F_a}$, where $F_a=-2at-kg$. One can again check that $F_a(t)$ is concave. And $F_a'<F_{a_0}'$ for $a>a_0$. Now within the annulus $t_{a_0}\leq t\leq t_{a_0}+\epsilon_{t,a}$, the integral of $\rho_a$ is smaller than $\rho_{a_0}$. So we must have $\rho_a<\rho_{a_0}$ at $t=t_{a_0}+\epsilon_{t,a}$. Therefore $\rho_a<\rho_{a_0}$ for all $t\geq t_{a_0}+\epsilon_{t,a}$. And the  lemma is proved.
\end{proof}
\begin{rem}
	This lemma allows us to calculate $\mu_{a,i}$ for $a$ small, using just our estimations for $I_b$ with $b<<k$ without worrying about the bigger $b$. So in the following, we will always assume $b<<k$. And in the summations, we often ignore the upper and lower indexes, keeping in mind that the numbers of terms that matter are $<<k$. 
\end{rem}

We now calculate 
\begin{eqnarray*}
\mu_{a,i}&=&\int \frac{|\bar{s}_{a,i}|^2_h}{\bar{\rho}}\frac{\omega_{FS}^{n+1}}{(n+1)!}\\
&=&\int \frac{|\bar{s}_{a,i}|^2_h}{\overline{\rho_a}}\frac{\overline{\rho_a}}{\bar{\rho}}\frac{\omega_{FS}^{n+1}}{(n+1)!}\\
&=&\int \frac{|s_{a,i}|^2_{h_M}}{\rho_a}\frac{\overline{\rho_a}}{\bar{\rho}}\frac{\omega_{FS}^{n+1}}{(n+1)!}
\end{eqnarray*}

As the Bergman kernel of $H^0(D, kA-aD)$, $\rho_a$ has asymptotic expansion
$$\rho_a=(N-a)^n[1+\frac{S_M}{2}\frac{1}{N-a}+O(\frac{1}{(N-a)^2})]$$
Recall that  $S_M$ is constant.

Given local frame $e_L$, we can write $$s_{a,i}=f_{a,i}e_L^{N-a}$$
We will denote the sum by
$$p_a=\sum_i|f_{a,i}|^2$$
So $\rho_a=p_ae^{-(N-a)\phi_M}$. And let $$\Psi_k=\sum x^{a-1}\frac{p_a}{I_a}$$
where $x=|z|^2$. Then $$\omega_{FS}=i\p\bp \log\Psi_k$$

Now in order to calculate $\mu_{a,i}$, we play the same trick by letting
$$r_a=\frac{I_a\Psi_k}{p_a}=\sum \frac{p_b}{p_a}\frac{I_a}{I_b}x^{b-1}$$

Then $$\omega_{FS}=i\p\bp \log r_a+i\p\bp \log p_a$$
We have $$i\p\bp \log p_a=\omega_{D,N-a}=(N-a)\omega_M+O(\frac{1}{(N-a)^2})$$
We use $y$ as local coordinates for $D$, then we have 
\begin{eqnarray*}
\omega_{FS}^{n+1}&=&i^{n+1}[\pz\bpz\log r_a+\py\bpy\log (r_ap_a)+(\pz\bpy+\py\bpz)\log r_a]^{n+1}\\
&=&(n+1)\pz\bpz\log r_a\wedge (\py\bpy\log (r_ap_a))^n\\
&+&(n+1)n(\pz\bpy\log r_a \wedge \py\bpz\log r_a)
\wedge  (\py\bpy\log (r_ap_a))^{n-1}
\end{eqnarray*}
At point $p\in D$, we can choose local frame to make $\phi_M(p)=0$ and $\p \phi_M(p)=0$, then 
$$\frac{p_{a+b}}{p_a}=\frac{\rho_{b+a}}{\rho_a}=(1-\frac{b}{N-a})^n(1+\frac{S_M}{2}\frac{b}{(N-a)^2})+O(\frac{1}{(N-a)^2})$$
and $$\py \frac{p_{a+b}}{p_a}=O(\frac{1}{(N-a)^2})$$
$$i\py\bpy \frac{p_{a+b}}{p_a}=(1-\frac{b}{N-a})^n(1+\frac{S_M}{2}\frac{b}{(N-a)^2}+O(\frac{1}{(N-a)^2}))(-b\omega_M)$$
By our estimations on the relations between the $I_a$'s (lemma \ref{2 terms} and lemma \ref{3 terms}), considered along the fibre over $p$, $r_a$ is very similar to the power series of $x$ we got in the one-dimensional model case. 
In particular, when $x$ is small enough, the value of $r_a$ is dominated by at most $3$ terms. 
We also have
$$\frac{\overline{\rho_a}}{\bar{\rho}}=\frac{x^{a-1}}{\sum \frac{I_a}{I_b}\frac{p_b}{p_a}x^{b-1}}=\frac{x^{a-1}}{r_a}$$
which, when $a$ is small, achieves its maximum at $x_a$ and then decay rapidly away from $x_a$.

Since $a<<k$, we have $O(\frac{1}{(N-a)^2})=O(\frac{1}{N^2})$.
Therefore,
$$\py r_a=\sum O(\frac{1}{N^2})\frac{I_ax^{b-1}}{I_b}$$
and 
$$i\py\bpy r_a=\sum -\frac{I_ac\omega_M}{I_{a+c}}[1+O(\frac{a}{N})]x^{a+c-1}$$ 
So 
\begin{eqnarray*}
	i\py\bpy \log r_a&=&\frac{1}{r_a}\sum -\frac{I_ac\omega_M}{I_{a+c}}[1+O(\frac{a}{N})]x^{a+c-1}\\
	&=&-\frac{\sum \frac{I_ac\omega_M}{I_{a+c}}[1+O(\frac{a}{N})]x^{c}}{\sum \frac{I_a}{I_{a+c}}[1+O(\frac{a}{N})]x^{c}}\\
	&=&o(N)\omega_M
\end{eqnarray*}
the reason for the last equality is that when $c>0$, $x^c$ is small, and when $c<0$, $c\geq -a$.

So $$i\py\bpy \log (r_ap_a)=(1+o(N))\omega_M$$
\begin{lem}
$(\pz\bpy\log r_a \wedge \py\bpz\log r_a)$ is $O(\frac{1}{k^2})$ compared to $\pz\bpz\log r_a\wedge \py\bpy\log (r_ap_a)$
\end{lem}
\begin{proof}
We have $$\pz r_a=\sum \frac{p_{c}}{p_a}\frac{I_a(c-1)\bar{z}}{I_{c}}x^{c-2}$$
So $$\bpz\py \log r_a=\frac{z}{r_a^2}\sum bO(\frac{1}{N^2})\frac{I_a^2}{I_bI_c}x^{b+c-3}$$
So $$\pz\bpy\log r_a\wedge \bpz\py\log r_a=\beta_1\cdot \beta_2$$
where $$\beta_1=\frac{1}{r_a^2}\sum bO(\frac{1}{N^2})\frac{I_a^2}{I_bI_c}x^{b+c-3}$$
and $$\beta_2=\frac{1}{r_a^2}\sum bO(\frac{1}{N^2})\frac{I_a^2}{I_bI_c}x^{b+c-2}$$
and we also ignored the differential forms.
Now we compare $\beta_1$ and $\pz\bpz \log r_a$.
We have
\begin{eqnarray*}
\pz\bpz \log r_a&=&\frac{1}{r_a^2}(r_a\Delta_z r_a-(\pz r_a)^2)\\
&=&\frac{1}{r_a^2}\sum [(c-1)^2-(b-1)(c-1)]\frac{p_bp_cI_a^2}{p_a^2I_bI_c}x^{b+c-3}	
\end{eqnarray*}

Since all summations are over $b$ and $c$, we can collect the coefficients.
\begin{eqnarray*}
\pz\bpz \log r_a&=&\frac{1}{r_a^2}\sum_l c_lx^{l-3}\\
\beta_1&=&\frac{1}{r_a^2}\sum O(\frac{1}{N^2})c_l'x^{l-3}
\end{eqnarray*}
where $c_l=\sum_{b,c}[(c-1)^2-(b-1)(c-1)]\frac{p_bp_cI_a^2}{p_a^2I_bI_c}$ and $c_l'=\sum_{b,c}b\frac{I_a^2}{I_bI_c}$

So $\beta_1$ is $O(\frac{b}{N^2})=O(\frac{1}{N})$ compared to $\pz\bpz \log r_a$.

and $|\beta_2|$ is $O(\frac{1}{N})$. And the lemma is proved
\end{proof}
\begin{rem}
A deeper argument should be able to show that the mixed term is of a much smaller magnitude, but for our purpose, $O(\frac{1}{k^2})$ is enough.
\end{rem}
Now we have 
$$\omega_{FS}^{n+1}=(1+O(\frac{1}{k^2}))(n+1)\pz\bpz\log r_a\wedge (\py\bpy\log (r_ap_a))^n$$
\subsection{Inside}
According to lemma \ref{3 terms}, by "inside", we mean when $a=O(\frac{\sqrt{k}}{\log k})$. So there is a $\lambda>0$ such that 
$$\frac{k}{a^2}>\lambda (\log k)^2$$

Let $F_a(x)=\frac{r_a}{x^{a-1}}$, and $q_c=\frac{I_a}{I_{a+c}}\frac{p_{a+c}}{p_a}$
\begin{lem}[Three terms Lemma]
We only need the three terms $q_{-1}x^{-1}+1+q_1x$ to compute $\mu_{a,i}$.
\end{lem}
\begin{proof}
Since the total measure of $\omega_\phi^{n+1}$ is $O(k^{n+1})$, we can ignore the part when $1/F_a(x)=\epsilon(k)$.

But when $$q_1x=e^{\lambda/2 (\log k)^2}$$ 
we have $$\frac{q_2}{q_1}x\approx (\frac{a+2}{a+1})^k(\frac{a}{a+1})^ke^{\lambda/2 (\log k)^2}<e^{-\lambda/2 (\log k)^2}$$ 
namely, as $x$ gets bigger away from $x_a$, before $q_2x^2$ can play any role, $1/F_a(x)$ is already $\epsilon(k)$. The argument is the same for $q_c$ with $c\leq -2$.
\end{proof}

Following the same line of arguments, we see that $$i\py\bpy\log r_a=O(1)\omega_M$$
Therefore we have 
\begin{equation}\label{fubini-study1}
\omega_{FS}^{n+1}=(1+O(\frac{1}{k}))(n+1)\pz\bpz\log r_a\wedge ((N-a)\omega_M)^n
\end{equation}
So in the integral for $\mu_{a,i}$, we can integrate along the fibres and then over $D$ as before.
\begin{prop}
$$\frac{1}{2\pi}\int \frac{1}{F_a(x)} i\pz\bpz \log F_a(x)=1+\epsilon(k)$$
for $a>1$.
\end{prop}
\begin{proof}
We just need to use the three terms, which can be standardized to $Q(x)=1+ax+bx^2$, with $\frac{b}{a^2}=\epsilon(k)$. Then 
$$i\pz\bpz \log Q(x)=\frac{a + 4 bx + abx^2}{(1+ax+bx^2)^2}$$
Therefore 
\begin{eqnarray*}
\int_0^{\infty}\frac{ax(a + 4 bx + abx^2)}{(1+ax+bx^2)^3}dx&=&\frac{a(a \sqrt{d}-2b(\log (1+a/d)-\log (1-a/d))) }{d^3}\\
&=&1+\epsilon(k)
\end{eqnarray*}
where $d=\sqrt{a^2-4b}$. 
\end{proof}
\begin{prop}
$$\frac{1}{2\pi}\int \frac{1}{F_1(x)} i\pz\bpz \log F_1(x)=\frac{1}{2}+\epsilon(k)$$
for $a>1$.
\end{prop}
\begin{proof}
We just need to use the two terms, which can be standardized to $Q(x)=1+ax$, with $\frac{1}{a}=\epsilon(k)$. Then 
$$i\pz\bpz \log Q(x)=\frac{a}{(1+ax)^2}$$
Therefore 
\begin{eqnarray*}
\int_0^{\infty}\frac{ax}{(1+ax)^3}dx&=&\frac{1}{2}
\end{eqnarray*}
\end{proof}
This means that the integrals along the fibres result in same constant with small errors. Then we can integrate over $D$.
Then by our choice of $s_{a,i}$, and the fact that the Bergman kernel $\rho_a$ is almost constant, we have the following theorem.
\begin{theorem}
$$\mu_{1,i}=\frac{1}{2}+O(\frac{1}{k})$$
and 
$$\mu_{a,i}=1+O(\frac{1}{k})$$ for $a>1$.
\end{theorem}

\subsection{neck}
By "neck", we mean when $a=O(\sqrt{k}\log k)$ and $a>C\frac{\sqrt{k}}{\log k}$ for some $C>0$.

We need to calculate the first several terms of $\log I_a$, as a function of $a$, at a point $a_0$.
For simplicity, in the following formula we will not write it but for every derivative of $I$ on the right hand side of the equation, we mean that it is evaluated at $a_0$.
\begin{eqnarray*}
\log I_a &=&\log I + c_1(a-a_0)+c_2(a-a_0)^2+c_3(a-a_0)^3\\
&+&c_4(a-a_0)^4+O((a-a_0)^5)
\end{eqnarray*}
where
\begin{eqnarray*}
c_1&=&\frac{I'}{I}\\
c_2&=&\frac{-(I')^2+II''}{2I^2}\\
c_3&=&\frac{1}{6I^3}[2(I')^3-3II''I'+I^2I^{(3)}]\\
c_4&=&\frac{1}{24I^4}[-6(I')^4+12I(I')^2I''-3I^2(I'')^2-4I'I^2I^{(3)}+I^3I^{(4)}]
\end{eqnarray*}

We have $$I^{(m)}=(-2)^m\int t^me^{E_a}dt$$
\begin{rem}
 As before, the integrals can be considered as over an $\epsilon_t$-neighborhood of $t_a$. So we omitted the limits to have some flexibility.
\end{rem}
We will play two small tricks here.

First, we can absorb the $I$'s in the denominators into the integrals of $I^{(m)}$, namely 
$$\frac{I^{(m)}}{I}=(-2)^m\int t^m \frac{e^{E}}{I}dt$$
We will write $d\mu=\frac{e^{E}}{I}dt$, which is so a probability measure, namely of total measure $1$.

Second, since $d\mu$ is centered at $t_a$, we shift it to be centered at $0$, namely, we let $\bar{t}=t-t_a$, then 
we have  
\begin{eqnarray*}
c_1&=&-2(t_a+i_1)\\
c_2&=&2(i_2-i_1^2)\\
c_3&=&\frac{-4}{3}[2i_1^3-3i_2i_1+i_3]\\
c_4&=&\frac{2}{3}[-6i_1^4+12i_1^2i_2-3i_2^2-4i_1i_3+i_4]
\end{eqnarray*}
where we use the notation $i_m=\int \bar{t}^md\mu$, and the new $d\mu$ is centered at $0$.
We need to estimate the $i_m$'s now.
The following standard Gaussian integrals  will be used.
\begin{lem}
\begin{eqnarray*}
\int_{-\infty}^{\infty}x^2e^{-bx^2}\sqrt{\frac{b}{\pi}}dx&=&\frac{1}{2b}\\
\int_{-\infty}^{\infty}x^4e^{-bx^2}\sqrt{\frac{b}{\pi}}dx&=&\frac{3}{4b^2}\\
\int_{-\infty}^{\infty}x^6e^{-bx^2}\sqrt{\frac{b}{\pi}}dx&=&\frac{15}{8b^3}\\
\int_{-\infty}^{\infty}x^8e^{-bx^2}\sqrt{\frac{b}{\pi}}dx&=&\frac{105}{16b^4}
\end{eqnarray*}
\end{lem}

The integrals of $i_m$ are from $-\epsilon_t$ to $\epsilon_t$, so we can use the Taylor series of $\bar{E_a}(\bar{t})=E_a(\bar{t}+t_a)$ to estimate the integrals. By our estimations of the derivatives of $E_a(t)$, we know that the dominating term is $\frac{E_a''(t_a)}{2}\bar{t}^2$. Because of the symmetry of this dominating term and the domain of the integrals, it is immediate to see that when $m$ is even $i_m$ is small compared to the case when $m$ is even. More precisely, in order to estimate $i_1$ and $i_3$, we have to use the degree three term in the Taylor series of $\bar{E_a}(\bar{t})$. Denoting $\frac{-E_a''(t_a)}{2}$ by $b$, we have
\begin{eqnarray*}
i_1&=&(1+O(\frac{1}{k}))C\frac{a^3}{k^2}\int_{-\infty}^{\infty}x^4e^{-bx^2}\sqrt{\frac{b}{\pi}}dx\\
 &=&(1+O(\frac{1}{k}))C\frac{a^3}{k^2}O(\frac{k^2}{a^4})\\
  &=& O(\frac{1}{a})
\end{eqnarray*}
\begin{eqnarray*}
i_3&=&(1+O(\frac{1}{k}))C\frac{a^3}{k^2}\int_{-\infty}^{\infty}x^6e^{-bx^2}\sqrt{\frac{b}{\pi}}dx
\\ &=&(1+O(\frac{1}{k}))C\frac{a^3}{k^2}O(\frac{k^3}{a^6})
\\ &=&O(\frac{k}{a^3})
\end{eqnarray*}
For $i_2$ and $i_4$ we only need to use the second degree term.
\begin{eqnarray*}
i_2&= &(1+O(\frac{a^4}{k^3}))\int_{-\infty}^{\infty}x^2e^{-bx^2}\sqrt{\frac{b}{\pi}}dx\\
 &=&(1+O(\frac{a^4}{k^3}))\frac{1}{2b}
\end{eqnarray*}
\begin{eqnarray*}
i_4&= & (1+O(\frac{a^4}{k^3}))\int_{-\infty}^{\infty}x^4e^{-bx^2}\sqrt{\frac{b}{\pi}}dx\\
 &=&(1+O(\frac{a^4}{k^3}))\frac{3}{4b^2}
\end{eqnarray*}
With these estimations, we can estimate
\begin{eqnarray}
c_1&=&(-2+O(\frac{1}{k}))t_a\\
c_2&=&(1+O(\frac{a^4}{k^3}+\frac{1}{k}))\frac{2}{-E_a''(t_a)}\\
c_3&=&O(\frac{k}{a^3})
\end{eqnarray}
For $c_4$, one easily sees that some terms are small:
$$-6i_1^4+12i_1^2i_2-4i_1i_3=O(\frac{1}{a^4}+\frac{k}{a^4}+\frac{k}{a^4})=O(\frac{k}{a^4})$$ 
But the two bigger terms cancel out
$$-3i_2^2+i_4=O(\frac{a^4}{k^3})O(\frac{k^2}{a^4})=O(\frac{1}{k})$$
Therefore, we have 
\begin{equation}
c_4=O(\frac{1}{k}+\frac{k}{a^4})
\end{equation}

With the preceding estimamtions, the remaining of the estimations for $\mu_{a,i}$ is exactly the same as that in one dimension \cite{SunSun}. Let us just sketch the steps.

We let $u=\log x$, then $f_a(u)=F_a(x)=\sum e^{\log q_c+cu}$. Clearly $f_a(u)$ is then a convex function, and has one minimum at $u_a\approx 2t_a$. Then we shift $u$ to make the minimum at $0$. Since $\log q_c\approx -\log I_{a+c}+\log I_a$, the Taylor series of $q_c$ is just the negative of that of $I_{a+c}$ at $a$. So $\log q_c$ is a concave function of $c$. After the shifting of $u$, the first order term of  $\log q_c$ is cancelled out, so  $\log q_c$ now achieve its maximum at $c=0$. 

Then we realize that $u$ only need to be restricted to $(-(\log k)^2,(\log k)^2)$. So then we only need to consider $|c|<(\log k)^5$. In particular 
$$i\py\bpy\log r_a=O((\log k)^5)\omega_M$$
So \begin{equation}\label{fubini-study2}
\py\bpy\log (r_ap_a)=(1+O(\frac{(\log k)^5}{k}))\py\bpy\log p_a
\end{equation}
Then the key argument is that because of the symmetry, the degree three term makes smaller contribution to the integral than the degree four term does. But we have $c_4=O(\frac{(\log k)^4}{k}$. 

All in all, we can use the standard function $$h_a(u)=\sum_{c\in \Z}e^{-\frac{kc^2}{2a^2}}e^{-cu}$$ to estimate the integral. Therefore we have the theorem 
\begin{theorem}
$$\mu_{a,i}=1+O(\frac{(\log k)^{60}}{k})$$.
\end{theorem}
Again, one can make the exponent $60$ smaller if one wish to.

\subsection{others}
For bigger $a$, since the mass of $\bar{s}_{a,i}$ lies outside $D_r$, we can use the Bergman kernel expansion 
$$\bar{\rho}=k^{n+1	}(1+\frac{c_0}{2}\frac{1}{k}+O(\frac{1}{k^2}))$$ 
to see that $\omega_{FS}=k\omega_\phi+O(\frac{1}{k^2})$. And then
\begin{eqnarray*}
\mu_{a,i}&=&\int \frac{|\bar{s}_{a,i}|^2_h}{\bar{\rho}}\frac{\omega_{FS}^{n+1}}{(n+1)!}\\
&=&\int |\bar{s}_{a,i}|^2_h\frac{(1+O(\frac{1}{k^2}))k^{n+1}}{k^{n+1	}(1+\frac{c_0}{2}\frac{1}{k}+O(\frac{1}{k^2}))}\frac{\omega_{\phi}^{n+1}}{(n+1)!}\\
&=&1-\frac{c_0}{2}\frac{1}{k}+O(\frac{1}{k^2})
\end{eqnarray*}

We then look at the off-diagonal entries. For simplicity, we write $\alpha=\{ a,i\}$ and $\beta=\{ b,j\}$, $\alpha\neq \beta$. 
Because of equations \ref{fubini-study1} and \ref{fubini-study2}, we see that $$\mu_{\alpha,\beta}=O(\frac{1}{k^2})$$ 
\subsection{Balancing Energy}
Now it is easy to estimate the balancing energy with parameter $\lambda=\frac{2}{3}$, $$|\underline{\mu}|^2_{2}$$
where $$\underline{\mu}=\mu_{\hat{L}}+\frac{1}{2}\mu_D-\frac{\vol \hat{L}+\frac{1}{2}\vol D}{\dim \hcal_k}I$$. 
\begin{theorem}
$$|\underline{\mu}|^2_{2}=O(\frac{(\log k)^{121}}{k^{1/2}})$$
\end{theorem}

\begin{proof}
$	O(\frac{(\log k)^{120}}{k^2})k^n\sqrt{k}\log k+O(\frac{1}{k^4})k^{n+1}=O(\frac{(\log k)^{120}}{k^{3/2-n}})$
\end{proof}

So we have proved our main theorem \ref*{thm1-1}.
\begin{cor}
	$(\hat{L},D,cA,0)$ is K-semistable.
\end{cor}
\begin{proof}
	We use directly the formula 2.2 in \cite{ssun}
	The Futaki invariant of any test configuration satisfies
	$$F\geq \lim_{k\to \infty}k^{-n-1}k^{(n+3)/2}\frac{(\log k)^{121}}{k^{3/4-n/2}}=0$$
\end{proof}

\bibliographystyle{plain}

\bibliography{references}

\end{document}